\documentclass[11pt]{article}

\usepackage{graphicx, amssymb, latexsym, amsfonts, amsmath, lscape, amscd,
amsthm, color, epsfig, mathrsfs, tikz, enumerate}

\newtheorem{theorem}{Theorem}[section]

\newtheorem{observation}[theorem]{Observation}

\newtheorem{lemma}[theorem]{Lemma}

\newcommand\DELETE[1]{}

\begin{document}

\title{{\bf On oriented cliques with respect to push operation}}
\author{
{\sc Julien Bensmail\footnote{The first author was supported by ERC Advanced Grant GRACOL, project no. 320812.} }$\,^{a}$, {\sc Soumen Nandi}$\,^{b}$, {\sc Sagnik Sen}$^{b}$\\
\mbox{}\\
{\small $(a)$ Technical University of Denmark, Denmark}\\
{\small $(b)$ Indian Statistical Institute, Kolkata, India}
}

\date{\today}

\maketitle

\begin{abstract}
 To push a vertex $v$ of a directed graph $\overrightarrow{G}$ is to change the 
  orientations  of all the arcs  incident with $v$. 
  An oriented graph is a directed graph without any cycle of length at most 2. 
An oriented clique is an oriented graph whose non-adjacent vertices are connected by a directed 2-path. A push clique is an oriented clique that remains an oriented clique even if one pushes any set of vertices of it. We show that it is NP-complete to decide if 
an undirected graph is underlying graph of a push clique or not. We also prove that 
a planar push clique can have at most 8 vertices. We also provide an exhaustive 
 list of minimal (with respect to spanning subgraph inclusion) planar push cliques. 
 
\medskip

\noindent \textit{Keyword:} oriented graphs, oriented cliques, push operation, planar graphs. 
\end{abstract}



\section{Introduction}
An {\textit{oriented graph}} $\overrightarrow{G}$
 is a directed graph with no cycle of length 1 or 2. The underlying undirected simple graph of $\overrightarrow{G}$ is
 denoted by $G$ while $\overrightarrow{G}$ is an \textit{orientation} of $G$. 
An \textit{oriented $k$-coloring} of an oriented graph $\overrightarrow{G}$ is a mapping $f$ from the vertex set $V(\overrightarrow{G})$ to a set of $k$ colors   such that  $f(u) \neq f(v)$ whenever $u$ and $v$ are adjacent and,   if 
$\overrightarrow{uv}$ and $\overrightarrow{wx}$ are 
two arcs in $\overrightarrow{G}$, then $f(u) = f(x)$ implies $f(v) \neq f(w)$. 
\noindent The \textit{oriented chromatic number} $\chi_o(\overrightarrow{G})$ of an oriented graph $\overrightarrow{G}$ is the smallest integer $k$ for which 
$\overrightarrow{G}$ admits an oriented $k$-coloring. Oriented coloring is a well studied topic (see the latest survey~\cite{sopena_updated_survey} for details).

 To \textit{push} a vertex $v$ of a directed graph $\overrightarrow{G}$ is to change the 
  orientations  of all the arcs (that is, to replace an arc $\overrightarrow{xy}$ by $\overrightarrow{yx}$) incident with $v$. 
  Push operation on directed graphs is a well studied topic~\cite{fisher-push, klostermeyer2, klostermeyer1}.

Two orientations  $\overrightarrow{G}$ and $\overrightarrow{G}'$ of $G$ are in a push relation if one can obtain $\overrightarrow{G}'$ by pushing a set of vertices of $\overrightarrow{G}$. 
The \textit{pushable chromatic number} $\chi_p(\overrightarrow{G})$, introduced by Klostermeyer and MacGillivray~\cite{push}, of an oriented graph $\overrightarrow{G}$ is the minimum oriented chromatic number taken over all
oriented graphs that are in  push relation with $\overrightarrow{G}$.

An oriented clique, introduced by Klostermeyer and MacGillivray~\cite{36}, is an oriented graph $\overrightarrow{C}$ with 
$\chi_o(\overrightarrow{C}) = |\overrightarrow{C}|$. 
In fact, an oriented clique is characterized as 
those oriented graphs in which each pair of non-adjacent vertices are connected by a directed 2-path. 
Due to the above characterization oriented cliques can be viewed as natural objects. 
Moreover, they 
play a significant role in studying oriented coloring as pointed out in~\cite{oclique_nandy}. 
An undirected simple graph is  called an \textit{underlying oriented clique} if it is underlying graph of an oriented clique.

We are interested in those oriented cliques that remain invarient under the push operation, that is, those  oriented cliques for which every oriented graph with a push relation with them are also oriented cliques. 
In fact it is easy to observe that if $\overrightarrow{C}$ is such a clique, 
then  $\chi_p(\overrightarrow{C}) = |\overrightarrow{C}|$. Thus, an oriented graph $\overrightarrow{C}$ is a \textit{push clique} if $\chi_p(\overrightarrow{C}) = |\overrightarrow{C}|$. 
Also an undirected simple graph is   an \textit{underlying push clique} if it is underlying graph of a push clique.

\begin{figure}

\centering
\begin{tikzpicture}[scale=.8]

\filldraw [black] (-.5,6) circle (2pt) {node[left]{}};

\node at (-.5,4.25) {$H_1$};

\filldraw [black] (.5,6.5) circle (2pt) {node[above]{}};
\filldraw [black] (.5,5.5) circle (2pt) {node[below left]{}};

\draw[-] (.5,6.5) -- (.5,5.5);

\node at (.5,4.25) {$H_2$};

\filldraw [black] (2,7) circle (2pt) {node[above]{}};
\filldraw [black] (1.5,6) circle (2pt) {node[below left]{}};
\filldraw [black] (2,5) circle (2pt) {node[below left]{}};

\draw[-] (2,7) -- (1.5,6) -- (2,5) -- (2,7);

\node at (1.7,4.25) {$H_3$};

\filldraw [black] (3,6.5) circle (2pt) {node[above]{}};
\filldraw [black] (3,5.5) circle (2pt) {node[below left]{}};
\filldraw [black] (4,6.5) circle (2pt) {node[below left]{}};
\filldraw [black] (4,5.5) circle (2pt) {node[below left]{}};

\draw[-] (3,6.5) -- (3,5.5);
\draw[-] (4,6.5) -- (4,5.5);
\draw[-] (3,6.5) -- (4,6.5);
\draw[-] (4,5.5) -- (3,5.5);

\node at (3.5,4.25) {$H_4$};

\filldraw [black] (5,5.25) circle (2pt) {node[above]{}};
\filldraw [black] (7,5.25) circle (2pt) {node[below left]{}};
\filldraw [black] (5,6.25) circle (2pt) {node[below left]{}};
\filldraw [black] (7,6.25) circle (2pt) {node[below left]{}};
\filldraw [black] (6,7.25) circle (2pt) {node[below left]{}};

\draw[-] (5,5.25) -- (7,5.25) -- (7,6.25) -- (6,7.25) -- (5,6.25) -- (5,5.25);
\draw[-] (5,6.25) -- (7,6.25);
\draw (6,7.25) .. controls (4.5,7.05) and (4.15,5.95) .. (5,5.25);

\node at (6,4.25) {$H_5$};

\filldraw [black] (9,7.5) circle (2pt) {node[above]{}};
\filldraw [black] (8,6.75) circle (2pt) {node[above]{}};
\filldraw [black] (10,6.75) circle (2pt) {node[above]{}};
\filldraw [black] (8,5.75) circle (2pt) {node[above]{}};
\filldraw [black] (10,5.75) circle (2pt) {node[above]{}};
\filldraw [black] (9,5) circle (2pt) {node[above]{}};

\draw[-] (9,7.5) -- (8,6.75) -- (8,5.75) -- (9,5) -- (10,5.75) -- (10,6.75) -- (9,7.5);
\draw[-] (8,5.75) -- (10,5.75);
\draw[-] (8,5.75) -- (10,6.75);
\draw (9,7.5) .. controls (11,7.25) and (10,5.75) .. (10, 5.75);
\draw (9,7.5) .. controls (11.25,7.5) and (11.25,5.5) .. (9,5);

\node at (9,4.25) {$H_6$};

\filldraw [black] (1,3.5) circle (2pt) {node[above]{}};
\filldraw [black] (0,2.5) circle (2pt) {node[above]{}};
\filldraw [black] (2,2.5) circle (2pt) {node[above]{}};
\filldraw [black] (0,1.5) circle (2pt) {node[above]{}};
\filldraw [black] (2,1.5) circle (2pt) {node[above]{}};
\filldraw [black] (1,.5) circle (2pt) {node[above]{}};

\draw[-] (1,3.5) -- (0,2.5) -- (0,1.5) -- (1,.5) -- (2,1.5) -- (2,2.5) -- (1,3.5);
\draw[-] (0,2.5) -- (2,2.5);
\draw[-] (0,1.5) -- (2,1.5);
\draw (1,3.5) .. controls (3,3) and (3,1) .. (1,.5);

\node at (1,-.25) {$H_7$};

\filldraw [black] (5,3.5) circle (2pt) {node[above]{}};
\filldraw [black] (4,2.5) circle (2pt) {node[above]{}};
\filldraw [black] (6,2.5) circle (2pt) {node[above]{}};
\filldraw [black] (4,1.5) circle (2pt) {node[above]{}};
\filldraw [black] (6,1.5) circle (2pt) {node[above]{}};
\filldraw [black] (5,.5) circle (2pt) {node[above]{}};

\draw[-] (5,3.5) -- (4,2.5) -- (4,1.5) -- (5,.5) -- (6,1.5) -- (6,2.5) -- (5,3.5);
\draw[-] (5,3.5) -- (6,1.5);
\draw[-] (5,3.5) -- (4,1.5);
\draw[-] (5,3.5) -- (5,.5);
\draw (4,2.5) .. controls (4.3,4.35) and (5.7,4.35) .. (6,2.5);

\node at (5,-.25) {$H_8$};

\filldraw [black] (9,3.5) circle (2pt) {node[above]{}};
\filldraw [black] (8,2.5) circle (2pt) {node[above]{}};
\filldraw [black] (10,2.5) circle (2pt) {node[above]{}};
\filldraw [black] (8,1.5) circle (2pt) {node[above]{}};
\filldraw [black] (10,1.5) circle (2pt) {node[above]{}};
\filldraw [black] (9,.5) circle (2pt) {node[above]{}};

\draw[-] (9,3.5) -- (8,2.5) -- (8,1.5) -- (9,.5) -- (10,1.5) -- (10,2.5) -- (9,3.5);
\draw[-] (8,1.5) -- (10,1.5);
\draw[-] (8,2.5) -- (10,2.5);
\draw (9,3.5) .. controls (8,3.5) and (7,2.7) .. (8,1.5);
\draw (10,2.5) .. controls (11.18,1.15) and (10,.7) .. (9,.5);
\node at (9,-.25) {$H_9$};

\filldraw [black] (1,-1) circle (2pt) {node[above]{}};
\filldraw [black] (0,-2) circle (2pt) {node[above]{}};
\filldraw [black] (2,-2) circle (2pt) {node[above]{}};
\filldraw [black] (0,-3) circle (2pt) {node[above]{}};
\filldraw [black] (2,-3) circle (2pt) {node[above]{}};
\filldraw [black] (.5,-4) circle (2pt) {node[above]{}};
\filldraw [black] (1.5,-4) circle (2pt) {node[above]{}};

\draw[-] (1,-1) -- (0,-2) -- (0,-3) -- (.5,-4) -- (1.5,-4) -- (2,-3) -- (2,-2) -- (1,-1);
\draw[-] (0,-2) -- (2,-3);
\draw[-] (0,-3) -- (1.5,-4);
\draw (2,-2) .. controls (2.7,-2.6) and (2.3,-3.8) .. (1.5,-4);
\draw (1,-1) .. controls (-.5,-1.4) and (-.3,-2.5) .. (0,-3);
\draw (1,-1) .. controls (-1.3,-1.2) and (-1,-3.8) .. (.5,-4);

\node at (1,-4.75) {$H_{10}$};

\filldraw [black] (5,-1) circle (2pt) {node[above]{}};
\filldraw [black] (4,-2) circle (2pt) {node[above]{}};
\filldraw [black] (6,-2) circle (2pt) {node[above]{}};
\filldraw [black] (4,-3) circle (2pt) {node[above]{}};
\filldraw [black] (6,-3) circle (2pt) {node[above]{}};
\filldraw [black] (4.5,-4) circle (2pt) {node[above]{}};
\filldraw [black] (5.5,-4) circle (2pt) {node[above]{}};

\draw[-] (5,-1) -- (4,-2) -- (4,-3) -- (4.5,-4) -- (5.5,-4) -- (6,-3) -- (6,-2) -- (5,-1);
\draw[-] (5,-1) -- (4,-3);
\draw[-] (5,-1) -- (6,-3);
\draw[-] (4,-3) -- (6,-3);
\draw[-] (4.5,-4) -- (6,-3);
\draw (6,-2) .. controls (6.7,-2.6) and (6.3,-3.8) .. (5.5,-4);
\draw (4,-2) .. controls (4.6,-.3) and (5.4,-.3) .. (6,-2);
\draw (4,-2) .. controls (3.4,-2.5) and (3.2,-5) .. (5.5,-4);

\node at (5,-4.75) {$H_{11}$};

\filldraw [black] (9,-1) circle (2pt) {node[above]{}};
\filldraw [black] (8,-2) circle (2pt) {node[above]{}};
\filldraw [black] (10,-2) circle (2pt) {node[above]{}};
\filldraw [black] (8,-3) circle (2pt) {node[above]{}};
\filldraw [black] (10,-3) circle (2pt) {node[above]{}};
\filldraw [black] (8.5,-4) circle (2pt) {node[above]{}};
\filldraw [black] (9.5,-4) circle (2pt) {node[above]{}};

\draw[-] (9,-1) -- (8,-2) -- (8,-3) -- (8.5,-4) -- (9.5,-4) -- (10,-3) -- (10,-2) -- (9,-1);
\draw[-] (9,-1) -- (10,-3);
\draw[-] (8,-2) -- (8.5,-4);
\draw[-] (8,-2) -- (9.5,-4);
\draw[-] (8,-2) -- (10,-3);
\draw (9,-1) .. controls (8,-1) and (7,-1.8) .. (8,-3);
\draw (9.5,-4) .. controls (10.2,-3.5) and (10.7,-3) .. (10,-2);
\draw (8.5,-4) .. controls (9.3,-4.6) and (11.8,-4) .. (10,-2);

\node at (9,-4.75) {$H_{12}$};

\filldraw [black] (1,-5.5) circle (2pt) {node[above]{}};
\filldraw [black] (0,-6.5) circle (2pt) {node[above]{}};
\filldraw [black] (2,-6.5) circle (2pt) {node[above]{}};
\filldraw [black] (0,-7.5) circle (2pt) {node[above]{}};
\filldraw [black] (2,-7.5) circle (2pt) {node[above]{}};
\filldraw [black] (.5,-8.5) circle (2pt) {node[above]{}};
\filldraw [black] (1.5,-8.5) circle (2pt) {node[above]{}};

\draw[-] (1,-5.5) -- (0,-6.5) -- (0,-7.5) -- (.5,-8.5) -- (1.5,-8.5) -- (2,-7.5) -- (2,-6.5) -- (1,-5.5);
\draw[-] (1,-5.5) -- (0,-7.5);
\draw[-] (1,-5.5) -- (2,-7.5);
\draw[-] (0,-7.5) -- (2,-7.5);
\draw[-] (0,-7.5) -- (1.5,-8.5);
\draw (1.5,-8.5) .. controls (2.2,-8) and (2.7,-7.5) .. (2,-6.5);
\draw (.5,-8.5) .. controls (1.3,-9.1) and (3.8,-8.5) .. (2,-6.5);
\draw (0,-6.5) .. controls (-1,-7.2) and (0,-8.8) .. (.5,-8.5);

\node at (1,-9.45) {$H_{13}$};

\filldraw [black] (5,-5.5) circle (2pt) {node[above]{}};
\filldraw [black] (4,-6.5) circle (2pt) {node[above]{}};
\filldraw [black] (6,-6.5) circle (2pt) {node[above]{}};
\filldraw [black] (4,-7.5) circle (2pt) {node[above]{}};
\filldraw [black] (6,-7.5) circle (2pt) {node[above]{}};
\filldraw [black] (4.5,-8.5) circle (2pt) {node[above]{}};
\filldraw [black] (5.5,-8.5) circle (2pt) {node[above]{}};

\draw[-] (5,-5.5) -- (4,-6.5) -- (4,-7.5) -- (4.5,-8.5) -- (5.5,-8.5) -- (6,-7.5) -- (6,-6.5) -- (5,-5.5);

\draw[-] (5,-5.5) -- (4.5,-8.5);
\draw[-] (5,-5.5) -- (5.5,-8.5);
\draw[-] (5.5,-8.5) -- (6,-6.5);
\draw (4,-6.5) .. controls (4.6,-4.8) and (5.4,-4.8) .. (6,-6.5);
\draw (4,-7.5) .. controls (4,-9.5) and (6,-9.5) .. (6,-7.5);

\node at (5,-9.45) {$H_{14}$};

\filldraw [black] (8.5,-5.5) circle (2pt) {node[above]{}};
\filldraw [black] (9.5,-5.5) circle (2pt) {node[above]{}};
\filldraw [black] (8,-6.5) circle (2pt) {node[above]{}};
\filldraw [black] (10,-6.5) circle (2pt) {node[above]{}};
\filldraw [black] (8,-7.5) circle (2pt) {node[above]{}};
\filldraw [black] (10,-7.5) circle (2pt) {node[above]{}};
\filldraw [black] (8.5,-8.5) circle (2pt) {node[above]{}};
\filldraw [black] (9.5,-8.5) circle (2pt) {node[above]{}};

\draw[-] (8.5,-5.5) -- (8,-6.5) -- (8,-7.5) -- (8.5,-8.5) -- (9.5,-8.5) -- (10,-7.5) -- (10,-6.5) -- (9.5,-5.5) -- (8.5,-5.5);
\draw[-] (8,-6.5) -- (8.5,-8.5) -- (10,-7.5) -- (9.5,-5.5) -- (8,-6.5);
\draw (9.5,-8.5) .. controls (10.2,-8) and (10.7,-7.5) .. (10,-6.5);
\draw (10,-6.5) .. controls (10.3,-5) and (8.8,-5) .. (8.5,-5.5);
\draw (8.5,-5.5) .. controls (8,-5.5) and (7,-6.3) .. (8,-7.5);
\draw (8,-7.5) .. controls (8,-9.3) and (9,-8.8) .. (9.5,-8.5);

\node at (9,-9.45) {$H_{15}$};

\filldraw [black] (.5,-10.5) circle (2pt) {node[above]{}};
\filldraw [black] (1.5,-10.5) circle (2pt) {node[above]{}};
\filldraw [black] (0,-11.5) circle (2pt) {node[above]{}};
\filldraw [black] (2,-11.5) circle (2pt) {node[above]{}};
\filldraw [black] (0,-12.5) circle (2pt) {node[above]{}};
\filldraw [black] (2,-12.5) circle (2pt) {node[above]{}};
\filldraw [black] (.5,-13.5) circle (2pt) {node[above]{}};
\filldraw [black] (1.5,-13.5) circle (2pt) {node[above]{}};

\draw[-] (.5,-10.5) -- (0,-11.5) -- (0,-12.5) -- (.5,-13.5) -- (1.5,-13.5) -- (2,-12.5) -- (2,-11.5) -- (1.5,-10.5) -- (.5,-10.5);
\draw[-] (.5,-10.5) -- (2,-11.5) -- (0,-11.5) -- (2,-12.5) -- (0,-12.5) -- (1.5,-13.5);
\draw (1.5,-13.5) .. controls (2.2,-13) and (2.7,-12.5) .. (2,-11.5);
\draw (.5,-10.5) .. controls (0,-10.5) and (-1,-11.3) .. (0,-12.5);
\draw (.5,-13.5) .. controls (2.5,-14.8) and (3.4,-12.1) .. (2,-11.5);
\draw (0,-12.5) .. controls (-1.5,-11.5) and (-.5,-9.5) .. (1.5,-10.5);

\node at (1,-14.25) {$H_{16}$};

\filldraw [black] (4.5,-10.5) circle (2pt) {node[above]{}};
\filldraw [black] (5.5,-10.5) circle (2pt) {node[above]{}};
\filldraw [black] (4,-11.5) circle (2pt) {node[above]{}};
\filldraw [black] (6,-11.5) circle (2pt) {node[above]{}};
\filldraw [black] (4,-12.5) circle (2pt) {node[above]{}};
\filldraw [black] (6,-12.5) circle (2pt) {node[above]{}};
\filldraw [black] (4.5,-13.5) circle (2pt) {node[above]{}};
\filldraw [black] (5.5,-13.5) circle (2pt) {node[above]{}};

\draw[-] (4.5,-10.5) -- (4,-11.5) -- (4,-12.5) -- (4.5,-13.5) -- (5.5,-13.5) -- (6,-12.5) -- (6,-11.5) -- (5.5,-10.5) -- (4.5,-10.5);
\draw[-] (4.5,-10.5) -- (6,-11.5) -- (4,-11.5) -- (6,-12.5) -- (4,-12.5) -- (5.5,-13.5);
\draw (5.5,-13.5) .. controls (6.2,-13) and (6.7,-12.5) .. (6,-11.5);
\draw (4.5,-10.5) .. controls (4,-10.5) and (3,-11.3) .. (4,-12.5);
\draw (4.5,-13.5) .. controls (2.5,-12.5) and (3,-10.5) .. (4.5,-10.5);
\draw (5.5,-13.5) .. controls (7.2,-12.9) and (7,-11.5) .. (5.5,-10.5);

\node at (5,-14.25) {$H_{17}$};

\filldraw [black] (8.5,-10.5) circle (2pt) {node[above]{}};
\filldraw [black] (9.5,-10.5) circle (2pt) {node[above]{}};
\filldraw [black] (8,-11.5) circle (2pt) {node[above]{}};
\filldraw [black] (10,-11.5) circle (2pt) {node[above]{}};
\filldraw [black] (8,-12.5) circle (2pt) {node[above]{}};
\filldraw [black] (10,-12.5) circle (2pt) {node[above]{}};
\filldraw [black] (8.5,-13.5) circle (2pt) {node[above]{}};
\filldraw [black] (9.5,-13.5) circle (2pt) {node[above]{}};

\draw[-] (8.5,-10.5) -- (8,-11.5) -- (8,-12.5) -- (8.5,-13.5) -- (9.5,-13.5) -- (10,-12.5) -- (10,-11.5) -- (9.5,-10.5) -- (8.5,-10.5);
\draw[-] (8.5,-10.5) -- (8,-12.5) -- (9.5,-10.5) -- (8.5,-13.5);
\draw[-] (9.5,-10.5) -- (9.5,-13.5) -- (10,-11.5);
\draw (10,-11.5) .. controls (10.3,-10) and (8.8,-10) .. (8.5,-10.5);
\draw (8,-11.5) .. controls (7.9,-11.5) and (7,-12.6) .. (8.5,-13.5);
\draw (8.5,-13.5) .. controls (8.9,-13.7) and (10,-14.5) .. (10,-12.5);
\draw (8,-11.5) .. controls (8,-9) and (11.3,-9.5) .. (10,-12.5);

\node at (9,-14.25) {$H_{18}$};

\end{tikzpicture}

\caption{List of minimal planar underlying push cliques upto spanning subgraph inclusion.}\label{fig planar push cliques}
\end{figure}
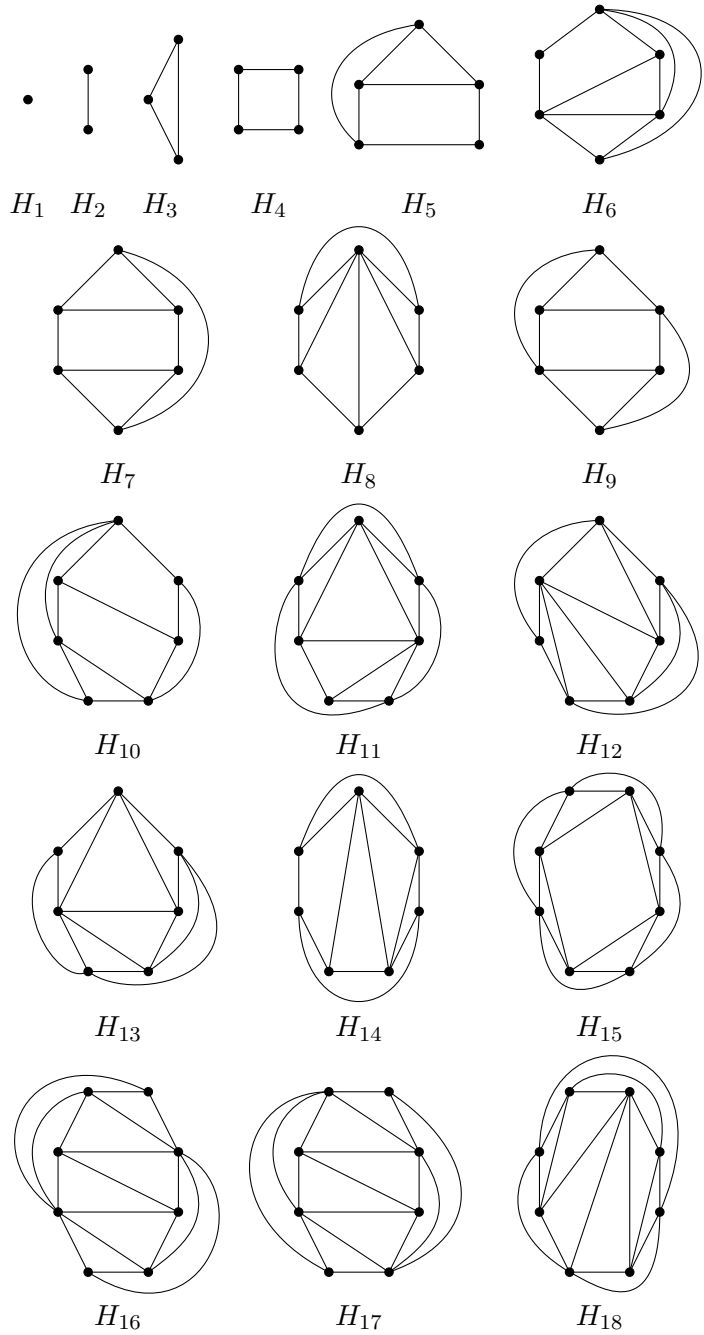

Given an undirected simple graph it is NP-hard to determine if it is an underlying oriented clique~\cite{bensmail2013complexity}. We prove an analogous result for underlying push cliques.

\begin{theorem}\label{th NP}
Let $G$ be an undirected simple graph. It is NP-complete to decide if $G$ is an underlying push clique. 
\end{theorem} 

Oriented cliques of planar and outerplanar graphs are studied in details~\cite{oclique_nandy}. 
It is proved that a planar oriented clique can have at most 15 vertices~\cite{oclique_nandy} which positively settled a conjecture by Klostermeyer and MacGillivray~\cite{36}. Note that there exists a planar oriented clique with 15 vertices which implies that the above mentioned bound is tight. 
Here we prove that a push clique can have at most 8 vertices and that this bound is tight.

\begin{theorem}\label{th planar push clique}
A planar push clique can have at most eight vertices. Moreover, there exists a planar push clique with eight vertices. 
\end{theorem}

Klostermeyer and MacGillivray showed that an outerplanar oriented clique can have at most seven vertices and any outerplanar oriented clique must 
have a particular oriented clique as a spanning subgraph~\cite{36}. 
Later this result was extended by providing an explicit list of eleven outerplanar graphs and proving that any outerplanar underlying 
oriented clique must have one of the eleven outerplanar graphs as its spanning subgraph~\cite{oclique_nandy}. In the same article the following question was asked: 
``Characterize the set $L$ of graphs such that any planar graph is an underlying oriented clique if and only if it contains one of the graphs from $L$ as a spanning subgraph.''  
Here we answer an analogous version of the above question for planar underlying push cliques.

\begin{theorem}\label{th planar push clique list}
An undirected  planar graph $G$ is an underlying push clique   if and only if it contains one of the eighteen planar graphs depicted in 
Figure~\ref{fig planar push cliques} as a spanning subgraph.
\end{theorem}

 In Section~\ref{sec preliminary} we introduce some basic  definitions and notations. The proofs of 
 Theorem~\ref{th NP},~\ref{th planar push clique} and~\ref{th planar push clique list} are given in 
 Section~\ref{sec NP},~\ref{sec bounds} and~\ref{sec list}, respectively. Theorem~\ref{th planar push clique}  was published 
in EuroComb 2013~\cite{EC_sen}.

\section{Preliminaries}\label{sec preliminary}
For an oriented  graph $G$  every parameter we introduce below is denoted using $G$  as a subscript. In order to simplify notation, this subscript will be dropped whenever there is no chance of confusion.

 The set of all adjacent vertices of a vertex $v$ in an oriented graph $G$ is called its set of \textit{neighbors} and is denoted by $N_{G}(v)$. 
 If there is an arc $uv$, then $u$ is an \textit{in-neighbor} of $v$ and $v$ is an \textit{out-neighbor} of $u$. 
 The set of all in-neighbors and the set of all out-neighbors of $v$ are denoted by $N_{G}^-(v)$ 
  and $N_{G}^+(v)$, respectively. 
The \textit{degree} of a vertex $v$ in an oriented graph $G$, denoted by 
$d_{G}(v)$, is the number of neighbors of $v$ in $ G$. 
Naturally, the  \textit{in-degree} (resp. \textit{out-degree}) of a vertex $v$ in an oriented graph $ G$, denoted by 
$d^-_{G}(v)$ (resp. $d^+_{G}(v)$), is the number of in-neighbors (resp. out-neighbors) of $v$ in $ G$.
The \textit{order} $ |G|$ of an oriented graph $ G$ is the cardinality of its
set of vertices $V(G)$.

Two vertices $u$ and $v$ of an oriented graph \textit{agree} on a third vertex $w$ of that graph if $w  \in N^\alpha(u) \cap N^\alpha(v)$ for some 
$ \alpha \in \{+,-\}$. 
Two vertices $u$ and $v$ of an oriented graph \textit{disagree} on a third vertex $w$ of that graph if $w  \in N^\alpha(u) \cap N^\beta(v)$ for some $ \{ \alpha, \beta \} = \{+,-\}$.

Now, take an oriented cycle of length 4 with arcs $ab, bc, cd, ad$. Note that all the oriented graphs which are in push relation with it are isomorphic to it. We call this a \textit{special 4-cycle}.
Notice that the non-adjacent vertices of a special 4-cycle always get different colors as they are always connected with a 2-dipath, no matter which vertex of the graph you push. This is in fact a necessary and sufficient condition for two non-adjacent vertices to receive two distinct colors under any oriented coloring with respect to any push relation.
Thus push cliques can be characterized as those oriented graphs whose any two distinct vertices are 
  either adjacent or part of a special 4-cycle.

\section{Proof of Theorem~\ref{th NP}}\label{sec NP}
Given an oriented graph one can check if it is a push clique or not in polynomial time using the above mentioned characterization of a push clique.

 Let $G$ be a graph. Define $G_*$ to be the graph obtained by adding a vertex $v_*$ to $G$ such that $v_*$ is adjacent to each vertex of $G$. 
Then  the following holds.

\begin{lemma}
The graph $G_*$ is an underlying push clique if and only if $G$ is an underlying oriented clique. 
\end{lemma}

\begin{proof}
Assume that $G_*$ is an underlying push clique. 
Let $\overrightarrow{G_*}$ be  an orientation of $G_*$ such that $\overrightarrow{G_*}$ is a  push clique. Let 
$\overrightarrow{G_*} ^{'}$ be the orientation of $G_*$ obtained by pushing all the in-neighbors of $v_*$. 
Therefore, in $\overrightarrow{G_*} ^{'}$ the vertices of $G$ are all out-neighbors of $v_*$. 
As $\overrightarrow{G_*} ^{'}$  is a  push clique, each pair of non-adjacent vertices of $\overrightarrow{G_*} ^{'}$ must agree on at least one vertex and must disagree on at least one vertex. Note that any pair of non-adjacent vertices must be vertices of $G$ and they agree on $v_*$. 
Thus, they must disagree on a vertex of $G$. Hence, the oriented graph induced by the vertices of $G$ obtained from $\overrightarrow{G_*} ^{'}$
is an oriented clique. Thus, $G$ is an underlying oriented clique. 

On the other hand, assume that $G$ is an underlying oriented clique. 
Let $\overrightarrow{G}$ be an orientation of $G$ such that  
$\overrightarrow{G}$ is an oriented clique. Now consider a  orientation $\overrightarrow{G_*}$ of $G_*$ in which every vertex of $G$ is an 
out-neighbor of $v_*$ and 
the oriented graph induced by the vertices of $G$ from $\overrightarrow{G_*}$ is isomorphic to $\overrightarrow{G}$. It is easy to observe that 
$\overrightarrow{G_*}$ is a push clique. 
\end{proof}

It is known that determining if a graph is an underlying oriented clique is NP-hard~\cite{bensmail2013complexity}. Hence, by the above lemma
our result follows.

\section{Proof of Theorem~\ref{th planar push clique}}\label{sec bounds}
 The lower bound follows from the example of the planar push clique of order 
8 depicted in Fig.~\ref{fig push planarmax}($i$).

\begin{figure}

\centering
\begin{tikzpicture}


\filldraw [black] (0,0) circle (2pt) {node[below]{}};
\filldraw [black] (2.5,0) circle (2pt) {node[below]{}};
\filldraw [black] (5,0) circle (2pt) {node[below]{}};

\filldraw [black] (1,.5) circle (2pt) {node[below]{}};
\filldraw [black] (4,.5) circle (2pt) {node[below]{}};

\filldraw [black] (2,1) circle (2pt) {node[below]{}};
\filldraw [black] (3,1) circle (2pt) {node[below]{}};

\filldraw [black] (2.5,2) circle (2pt) {node[below]{}};

\draw[->] (0,0) -- (1,0);
\draw[->] (1,0) -- (4,0);
\draw[-] (4,0) -- (5,0);

\draw[->] (0,0) -- (.5,.25);
\draw[->] (.5,.25) -- (1.5,.75);
\draw[-] (1.5,.75) -- (2,1);

\draw[->] (5,0) -- (4.5,.25);
\draw[->] (4.5,.25) -- (3.5,.75);
\draw[-] (3.5,.75) -- (3,1);

\draw[->] (1,.5) -- (1.75,.25);
\draw[-] (1.75,.25) -- (2.5,0);

\draw[->] (2,1) -- (2.25,.5);
\draw[-] (2.25,.5) -- (2.5,0);

\draw[-<] (3,1) -- (2.75,.5);
\draw[-] (2.75,.5) -- (2.5,0);

\draw[-<] (4,.5) -- (3.25,.25);
\draw[-] (3.25,.25) -- (2.5,0);

\draw[->] (2.5,2) -- (1.25,1);
\draw[-] (1.25,1) -- (0,0);

\draw[->] (2.5,2) -- (1.75,1.25);
\draw[-] (1.75,1.25) -- (1,.5);

\draw[->] (2.5,2) -- (2.25,1.5);
\draw[-] (2.25,1.5) -- (2,1);

\draw[->] (2.5,2) -- (3.75,1);
\draw[-] (3.75,1) -- (5,0);

\draw[->] (2.5,2) -- (3.25,1.25);
\draw[-] (3.25,1.25) -- (4,.5);

\draw[->] (2.5,2) -- (2.75,1.5);
\draw[-] (2.75,1.5) -- (3,1);

\node at (2.5,-.5) {$(i)$};


\filldraw [black] (7,0) circle (2pt) {node[left]{}};
\filldraw [black] (7,2) circle (2pt) {node[left]{$a$}};

\filldraw [black] (7.5,1) circle (2pt) {node[right]{}};

\filldraw [black] (8,.5) circle (2pt) {node[below]{$b$}};
\filldraw [black] (8,1) circle (2pt) {node[right]{}};
\filldraw [black] (8,1.5) circle (2pt) {node[right]{}};

\filldraw [black] (8.5,1) circle (2pt) {node[right]{}};

\filldraw [black] (9,0) circle (2pt) {node[right]{}};
\filldraw [black] (9,2) circle (2pt) {node[right]{}};

\draw[-] (7,0) -- (7,2);
\draw[-] (7,0) -- (9,0);
\draw[-] (7,0) -- (8,.5);
\draw[-] (7,0) -- (7.5,1);

\draw[-] (9,2) -- (7,2);
\draw[-] (9,2) -- (9,0);
\draw[-] (9,2) -- (8,1.5);
\draw[-] (9,2) -- (8.5,1);

\draw[-] (8,1) -- (8,1.5);
\draw[-] (8,1) -- (8,.5);
\draw[-] (8,1) -- (8.5,1);
\draw[-] (8,1) -- (7.5,1);

\draw[-] (7,2) -- (7.5,1);
\draw[-] (7,2) -- (8,1.5);

\draw[-] (9,0) -- (8.5,1);
\draw[-] (9,0) -- (8,.5);

\node at (8,-.5) {$(ii)$};

\end{tikzpicture}

\caption{$(i)$ A planar push clique of order 8, $(ii)$ The unique diameter 2 planar graph with domination number 3~\cite{dom}.}\label{fig push planarmax}
\end{figure}
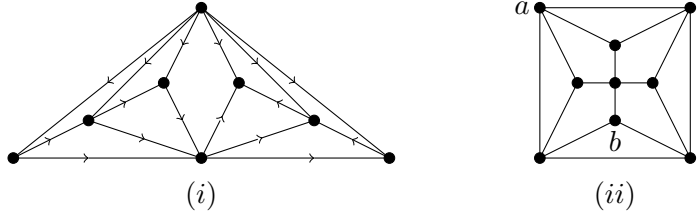

\medskip

Now we will prove the upper bound.
We know that a push clique is an oriented  graph with each pair of non-adjacent vertices 
agreeing on at least one other vertex and disagreeing on one other vertex.  In particular, a push clique  has diameter at most 2.

Goddard and Henning~\cite{dom} showed that every planar graph with diameter 2 has domination number at most 2 except for
a particular  graph on nine  vertices (see Fig.~\ref{fig push planarmax}($ii$)). It is easy check that vertices $a$ and $b$
in this graph   are not connected by two distinct 2-paths. Therefore, 
it is not a push clique. 
 Hence, a planar push cliques must have domination number at most 2.

Let $\overrightarrow{B}'$ be a planar push clique  dominated by the vertex $v$. 
Push all the in-neighbors of $v$ to obtain the oriented graph $\overrightarrow{B}$ from $\overrightarrow{B}'$. Note that $\overrightarrow{B}$
is a push clique with  $N^+_{\overrightarrow{B}}(v) = N_{\overrightarrow{B}}(v) = V(\overrightarrow{B}) \setminus \{v\}$.
Observe that each pair of non-adjacent vertices from $N_{\overrightarrow{B}}(v)$ agrees on $v$ and thus, must disagree on a vertex from 
$N_{\overrightarrow{B}}(v)$. Therefore, the graph induced by $N_{\overrightarrow{B}}(v)$ must be an oriented clique. Moreover, note that 
the oriented graph induced by $N_{\overrightarrow{B}}(v)$ is an outerplanar graph. We know that an outerplanar oriented clique can have at most 
seven vertices~\cite{sopena_updated_survey}. Thus, $\overrightarrow{B}$ has order at most eight.

To prove Theorem~\ref{th planar push clique} it will be enough to prove that  
any planar push clique with domination number 2 must have order at most 8. 
More precisely, we need to prove the following lemma.

\begin{lemma}\label{lem planar dom 2 implies order 8}
Let $ \overrightarrow{H}$ be a planar push clique with domination number 2.  Then $ |V(\overrightarrow{H}) | \leq 8$. 
\end{lemma}

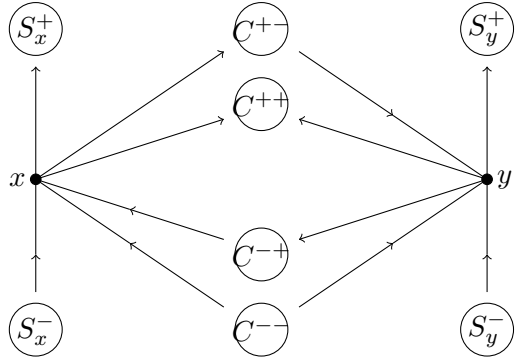
\begin{figure}

\centering
\begin{tikzpicture}

\draw [black] (3,0) circle (10pt) {node[]{$C^{--}$}};
\draw [black] (3,1) circle (10pt) {node[]{$C^{-+}$}};
\draw [black] (3,3) circle (10pt) {node[]{$C^{++}$}};
\draw [black] (3,4) circle (10pt) {node[]{$C^{+-}$}};

\draw [black] (0,0) circle (10pt) {node[]{$S_x^{-}$}};
\draw [black] (0,4) circle (10pt) {node[]{$S_x^{+}$}};

\draw [black] (6,0) circle (10pt) {node[]{$S_y^{-}$}};
\draw [black] (6,4) circle (10pt) {node[]{$S_y^{+}$}};

\filldraw [black] (0,2) circle (2pt) {node[left]{$x$}};
\filldraw [black] (6,2) circle (2pt) {node[right]{$y$}};

\draw[->] (0,.5) -- (0,1);
\draw[->] (0,1) -- (0,3.5);

\draw[->] (6,.5) -- (6,1);
\draw[->] (6,1) -- (6,3.5);

\draw[->] (0,2) -- (2.5,3.7);

\draw[->] (0,2) -- (2.5,2.8);

\draw[->] (2.5,1.2) -- (1.25,1.6);
\draw[-] (1.25,1.6) -- (0,2);

\draw[->] (2.5,.3) -- (1.25,1.15);
\draw[-] (1.25,1.15) -- (0,2);

\draw[->] (3.5,3.7) -- (4.75,2.85);
\draw[-] (4.75,2.85) -- (6,2);

\draw[->] (6,2) -- (3.5,2.8);

\draw[->] (6,2) -- (3.5,1.2);

\draw[->] (3.5,.3) -- (4.75,1.15);
\draw[-] (4.75,1.15) -- (6,2);

\end{tikzpicture}

\caption{Structure of $\vec{G}$ (not a planar embedding)}\label{orientable not emb}

\end{figure}	

Let $\overrightarrow{G}$ be a planar push clique with $|V(\overrightarrow{G})| > 8$. 
Assume that $\overrightarrow{G}$ 
is triangulated and has domination number 2.

We define the partial order $\prec$ for the set of all dominating sets of order 2 of $ \overrightarrow{G}$ as follows: for any two
dominating sets $D = \{x,y\}$  and $D' = \{ x',y' \}$ of order 2 of $ \overrightarrow{G}$,  
$D' \prec D$ if and only if $ |N_{\overrightarrow{G}}(x') \cap N_{\overrightarrow{G}}(y')| < |N_{\overrightarrow{G}}(x) \cap N_{\overrightarrow{G}}(y)| $.

Let $D = \{x,y\}$ be a maximal dominating set of order 2 of $ \overrightarrow{G}$ with respect to $\prec$. 
Also for the rest of this proof, $t,t^\prime,\alpha, \overline{\alpha}, \beta, \overline{\beta}$ are variables satisfying $\{t,t^\prime \} = \{x, y\}$ and $\{\alpha, \overline{\alpha} \} = \{\beta, \overline{\beta} \} = \{+,- \}$.

Now, we fix the following notations (Fig: \ref{orientable not emb}): 

\begin{align}\nonumber
&C = N_{\overrightarrow{G}}(x) \cap N_{\overrightarrow{G}}(y)\text{, } C^{\alpha \beta} = N^{\alpha}_{\overrightarrow{G}}(x) \cap N^{\beta}_{\overrightarrow{G}}(y)\text{, } C_t^{\alpha} = N^{\alpha}_{\overrightarrow{G}}(t) \cap C,\\
&S_t = N_{\overrightarrow{G}}(t) \setminus C\text{, } S^{\alpha}_t = S_t \cap N^{\alpha}_{\overrightarrow{G}}(t)\text{ and } S = S_x \cup S_y. \nonumber
\end{align}

Without loss of generality, assume that, the oriented graph $\overrightarrow{G}$  is such that
 $C_x = C^+_x$, $S_x = S^+_x$, $S_y = S^+_y$ and  $|C^+_y| \geq |C^-_y|$. 
 Note that, it is possible to obtain an alternative orientation of $G$ from 
$\overrightarrow{G}$ that satisfies the above conditions by pushing some vertices of $\overrightarrow{G}$. 
That is why our assumption is valid. 

Hence we have,

\begin{equation}\label{orientable order}
9 \leq | \overrightarrow{G} | = | D | + | C | + | S |.
\end{equation}
 
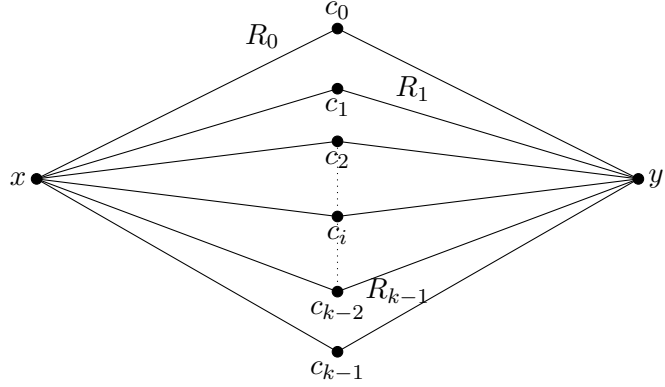
\begin{figure}

\centering
\begin{tikzpicture}

\filldraw [black] (0,2) circle (2pt) {node[left]{$x$}};
\filldraw [black] (8,2) circle (2pt) {node[right]{$y$}};

\filldraw [black] (4,4) circle (2pt) {node[above]{$c_0$}};

\filldraw [black] (4,3.2) circle (2pt) {node[below]{$c_1$}};

\filldraw [black] (4,2.5) circle (2pt) {node[below]{$c_2$}};

\filldraw [black] (4,1.5) circle (2pt) {node[below]{$c_i$}};

\filldraw [black] (4,.5) circle (2pt) {node[below]{$c_{k-2}$}};

\filldraw [black] (4,-.3) circle (2pt) {node[below]{$c_{k-1}$}};


\draw[-] (0,2) -- (4,-.3);
\draw[-] (0,2) -- (4,3.2);
\draw[-] (0,2) -- (4,4);
\draw[-] (0,2) -- (4,1.5);
\draw[-] (0,2) -- (4,2.5);
\draw[-] (0,2) -- (4,.5);

\draw[-] (8,2) -- (4,-.3);
\draw[-] (8,2) -- (4,3.2);
\draw[-] (8,2) -- (4,4);
\draw[-] (8,2) -- (4,1.5);
\draw[-] (8,2) -- (4,2.5);
\draw[-] (8,2) -- (4,.5);

\draw[dotted] (4,2.5) -- (4,.5);

\node at (3,3.9){$R_0$};

\node at (5,3.2){$R_1$};

\node at (4.8,.5){$R_{k-1}$};

\end{tikzpicture}

\caption{A planar embedding of $und( \vec{H})$ }\label{orientable undh}

\end{figure}

Let  $ \overrightarrow{H}$ be the oriented graph obtained from the induced subgraph $ \overrightarrow{G}[D \cup C]$ of $ \overrightarrow{G}$ by deleting all the arcs between the vertices of $D$ and  all the arcs 
 between the vertices of $C$.  Note that it is possible to extend  the planar embedding of $und(\overrightarrow{H})$ given in Fig~\ref{orientable undh}  to a planar embedding of $und(\overrightarrow{G})$ for some particular ordering of the elements of, say  $C = \{c_0, c_1, ..., c_{k-1} \} $.

Notice that $und( \overrightarrow{H})$ has $k$ faces, namely the unbounded face $F_0$ 
and the faces $F_{i}$ bounded by edges 
$xc_{i-1}, c_{i-1}y, yc_{i}, c_{i}x$ for $i \in \{ 1, ..., k-1 \}$. 
Geometrically, $und( \overrightarrow{H})$  divides the plane into $k$ connected components. 
The \textit{region}  $R_i$ of $ \overrightarrow{G}$  is the $i^{th}$ connected component (corresponding to the face $F_i$) of the plane.
\textit{Boundary points} of a region $R_i$ are $c_{i-1}$ and $c_i$ for $i \in \{ 1, ..., k-1 \}$ and, $c_0$ and $c_{k-1}$ for $i=0$. Two regions are \textit{adjacent} if  they have at least one common boundary point (hence, a region is adjacent to itself).

Now for the different possible values of $| C |$, we want to show that  $und( \overrightarrow{H})$ cannot be extended to a planar push clique of order at least $9$. Note that for extending $und( \overrightarrow{H})$ to $ \overrightarrow{G}$ we can add new vertices only from $S$. Any vertex $v \in S$ will be inside one of the regions $R_i$. If there is at least one vertex of $S$ in a region $R_i$, then $R_i$ is \textit{non-empty} and \textit{empty} otherwise. 
In fact,
when there is no chance of confusion,  $R_i$ might represent the set of vertices of $S$ contained in the region $R_i$.

First we will ask the question that ``How small $|C|$  can be?'' and prove the following lower bound of $|C|$.

\begin{lemma}
$|C| \geq 3$.
\end{lemma}

\begin{proof}
We know that $x$ and $y$ are either connected by two distinct 2-paths or by an arc. So, if $x$ and $y$ are 
non-adjacent, then we have $|C| \geq 2$. If $x$ and $y$ are adjacent, then  the triangulation of $\overrightarrow{G}$ implies $|C| \geq 2$.

To complete the proof we need to show that $|C| \neq 2$. We will prove by contradiction. Therefore, assume that 
$|C| = 2$. To get a contradiction to  our 
assumption, by equation~\ref{orientable order}, it will be enough to show
$|S| \leq 4$.

Note that, if we have $S_x = \emptyset $, then triangulation will force either  
mutiple edges $c_0c_1$ (one in $R_0$ and one in $R_1$) or the edge $xy$ making $x$ a dominating vertex. 
Both contradicts our assumption. Hence we do not have  $S_t = \emptyset $ for any $t \in \{x,y\}$.

Note that if for some $i \neq j$ both $S_x \cap R_i$ and $S_y \cap R_j$ are  non-empty, then the vertices of the two sets must be adjacent to both $c_0$ and $c_1$.   This will imply $|S_x \cap R_i| \leq 1$ and $|S_y \cap R_j| \leq 1$.

Thus, if all the four sets $S_t \cap R_i \neq \emptyset$ for all $(t,i) \in \{x,y\} \times \{0,1\}$, then $|S| \leq 4$.

Assume that we have exactly three non-empty sets $S_x \cap R_0$, $S_x \cap R_1$ and $S_y \cap R_0$ among the 
four sets  $S_t \cap R_i$ for all $(t,i) \in \{x,y\} \times \{0,1\}$. 
By triangulation we have the edge $c_0c_1$ inside
$R_1$ and there is at least one vertex in $|S_x \cap R_0|$
 adjacent to  $c_0$. 
 Now we have the dominating set $\{x, c_0\}$ with at 
 least three common neighbors 
 ($c_1$, a vertex from $S_x \cap R_0$ and a vertex from $S_x \cap R_1$) contradicting the maximality of $D$.

Therefore, exactly two sets among the four sets  $S_t \cap R_i$ for all $(t,i) \in \{x,y\} \times \{0,1\}$ 
can be non-empty.
As $S_x, S_y \neq \emptyset $ we must have $S_x \cap R_i$ and $S_y \cap R_j$ as non-empty sets for some $i,j \in \{0,1\}$. 
If $i \neq j$, then $|S| \leq 2$. Thus, we can assume that  
 exactly the  sets   $S_x \cap R_0$ and $S_y \cap R_0$
are the two non-empty sets among the four sets $S_t \cap R_i$ for all $(t,i) \in \{x,y\} \times \{0,1\}$.

For this case, assume that $S_x = \{x_1,x_2,...$ 
$ ,x_{n_x}\}$ and $S_y = \{y_1,y_2,...,y_{n_y}\}$. 
Without loss of generality also assume that we have the 
edges $c_0x_1, x_1x_2, ..$
$., x_{n_x - 1}x_{n_x}, x_{n_x}c_1$ and the edges 
$c_0y_1, y_1y_2, ..., y_{n_y - 1}y_{n_y}, y_{n_y}c_1$ by triangulation. 
Furthermore, we can assume $n_x \geq n_y$ without loss of generality. 

Assume $n_y = 1$. So, to have $|S| \geq 5$ we should have $n_x \geq 4$.
Note that we cannot have the edge $xy$ as otherwise $\{y_1,x\}$ will be a dominating set with at least three common neighbors $\{c_0,c_1,y\}$ contradicting the maximality of $D$. Hence, we have the edge 
$c_0c_1$ inside $R_1$ by triangulation. Note that the vertex  $x_2 \in S_x$ must be adjacent to 
either $c_0$ or $c_1$ or $y_1$  to have two distinct 2-paths  connecting it to $y$. This will create a
dominating set $\{c_0,x\}$ or $\{c_1,x\}$ or $\{y_1,x\}$ with at least three common neighbors 
$\{x_1,x_2,c_1\}$ or $\{x_{n_x},x_2,c_0\}$ or $\{x_2,c_0,c_1\}$ respectively. This will contradict the maximality of $D$. Therefore, 
 $n_y \geq 2$.

Now assume that we have the edge $x_2c_0$. Then to have  two distinct 2-paths  connecting a vertex $w \in S_y$  to $x_1$ we will have $y$ adjacent to both $x_2$ and $c_0$. That means, each vertex of $S_y$ will be adjacent to 
both $x_2$ and $c_0$. But this is not possible keeping the graph planar as $n_y \geq 2$. So, there is no edge between $c_0$ and $x_2$. By similar arguments, we can show that every $t_i$ is non-adjacent to $c_0$ for 
$i \in \{2,3,...,n_t\}$ and every $t_j$ is non-adjacent to $c_1$ for 
$i \in \{1,2,...,n_t-1\}$ for all $t \in \{x,y\}$. With a similar argument we can also show that the edge 
$t_it_{i+k}$ for $1 \leq i < i+k \leq n_t$ does not exist unless $k=1$ for any $t \in \{x,y\}$.

Now notice that $n_x \geq 3$ by equation~\ref{orientable order} and the assumption that $n_x \geq n_y$.
By triangulation we must have the edge $x_2y_i$ for some $i \in \{1,2,...,n_y\}$.  
 Then to have two distinct 2-paths between $x_1$ and $y_j$ for $j \in \{i+1,...,n_y\}$ and 
 to have two distinct 2-paths between $x_3$ and $y_l$ for $l \in \{1,...,i-1\}$ we must have every vertex of $S_y$ adjacent to $x_2$. 
 
 If $n_y \geq 3$, then we cannot have two distinct 2-paths  between the 
 non-adjacent vertices $x_1$ and $y_3$. So we must have
$n_y = 2$.

 Now to have two distinct 2-paths between  the 
 non-adjacent vertices $x_1$ and $y_2$ we must have the edge $x_1y_1$. This creates the dominating set $\{x,y_1\}$
 with at least three common neighbors $\{c_0,x_1,x_2\}$ contradicting the maximality of $D$. Therefore, it is not possible to have $|C| = 2$. 
 \end{proof}

Now we will ask the question that ``How big $|C|$  can be?'' and prove the following upper bound of $|C|$.

\begin{lemma}
$|C| \leq 3$.
\end{lemma}

\begin{proof}
First assume that $S = \emptyset$. Then $|C| = k \geq 7$ and there are at least four vertices of $C$ that agree with each other  on both $x$ and $y$. Among those four vertices, two much be such that they are non-adjacent and the only common neighbors they have are $x$ and $y$. Thus, those two non-adjacent vertices do not disagree on any vertex, a contradiction.

Now assume that $|C| = k \geq 4$ and $S \neq \emptyset$. Then, without loss of generality, assume a vertex 
 $v \in S_x \cap R_0$.  To have two distinct 2-paths connecting  $v$ to $c_1$ and $c_{k-2}$ we must have 
 the edges $vc_0$ and $vc_{k-1}$.  Hence we have $|S_x \cap R_0| \leq 1$. In fact with a similar argument we can show that 
 
 \begin{align}\label{eqn utko for c=4}
 |S_t \cap R_i| \leq 1 \text{ for all } (t,i) \in \{x,y\} \times \{0,1, ..., k-1 \}. 
 \end{align}

 Also note that if we have a vertex $v \in S_x \cap R_0$, then it is not possible to have any
  vertex in $S_y \cap R_i$ for all
$ i \in \{ 1,2,..., k-1 \}$ and in $S_x \cap R_j$ for all $ j \in \{ 1,2,..., k-1 \}$.  So basically, only adjacent regions can be non-empty.

Hence, at most two of the  sets $S_t \cap R_i$  for all $ (t,i) \in \{x,y\} \times \{0,1,...k-1 \}$  can be non-empty. 
Then equation~\ref{eqn utko for c=4} implies $|S| \leq 2$.  Then by equation~\ref{orientable order} we have
 
 \begin{align}\nonumber
 9 \leq |G| \leq  2 + 4 + 2 = 8.
 \end{align}
 
This is a contradiction.
 \end{proof}

Therefore, the only possible value for $|C|$ is 3. To prove Theorem~\ref{th planar push clique} we will show that 
$|C| = 3$ is not possible in the following lemma.

\begin{lemma}
$|C| \neq 3$.
\end{lemma}

\begin{proof}
We will prove this lemma by contradiction. So, assume that $|C| = 3$.  Also, without loss of generality, 
 assume that $|S_x| \geq |S_y|$.

Note that by equation~\ref{orientable order} we have $|S| \geq 5$. Hence we have $|S_x| \geq 3$. 

First assume that $S_y = \emptyset$. Hence we do not have the edge $xy$ as otherwise $x$ will 
dominate the whole graph. Now note that any two regions are adjacent for $|C| = 3$. 
The vertices from different regions must be adjacent to their unique common boundary point to have two distinct 2-paths connecting them. 

Hence, if we have all the regions non-empty, then we will have 
the vertices of each region adjacent to both the boundary points of that region. This will imply

\begin{align}\nonumber 
|S_x \cap R_i| \leq 1 \text{  for all } i \in \{0,1,2\}.
\end{align}

This will imply $|S| \leq 3$ and contradict our assumption. Hence, it is not possible to have all the three regions non-empty when $S_y = \emptyset$.

If we have exactly two regions, say $R_0$ and $R_1$,  non-empty, then every vertex of $S_x$ must be adjacent to 
$c_0$ to create two distinct 2-paths between the vertices of $S_x \cap R_0$ and the vertices of 
$S_x \cap R_1$. This will create a dominating set $\{c_0,x\}$ with at least four common neighbors contradicting the maximality of $D$. Hence, we can have at most one region non-empty when $S_y = \emptyset$.

Now assume that exactly one region, say $R_1$, is non-empty. Then each vertex of $S_x$ must be adjacent to 
either 
$c_0$ or $c_1$ to have two distinct 2-paths connecting it to $c_2$. Then, without loss of generality, we will have at least three vertices of $S_x$ adjacent to $c_0$. This will create a dominating set $\{c_0,x\}$ with at least four common neighbors (three vertices from $S_x$ and $c_2$ because of triangulation) contradicting the maximality of $D$. 

Hence $S_y \neq \emptyset$.

\medskip

Now assume, without loss of generality, that  $S_x \cap R_0 \neq \emptyset$. This implies 
$S_y \cap R_1 = \emptyset$ and $S_y \cap R_2 = \emptyset$ as it is not possible to have two distinct 2-paths between 
the vertices of $S_x$ from one region and the vertices of $S_y$ from a different region. But we also know that $S_y \neq \emptyset$. Hence we must have $S_y \cap R_0 \neq \emptyset$. 

Now assume that $S_x = \{x_1,x_2,...$ 
$ ,x_{n_x}\}$ and $S_y = \{y_1,y_2,...,y_{n_y}\}$. 
Without loss of generality also assume that we have the 
edges $c_0x_1, x_1x_2, .., x_{n_x - 1}x_{n_x},$
$ x_{n_x}c_1$ and the edges 
$c_0y_1, y_1y_2, ...,$
$ y_{n_y - 1}y_{n_y}, y_{n_y}c_1$ by triangulation. 

Now $x_2$ must be adjacent to either $c_0$ or $c_2$ for having two distinct 2-paths connecting $x_2$ and $c_1$. 
Without loss of generality assume that $x_2$ is adjacent to $c_0$. Now each vertex of $S_y$ must be adjacent to both $x_2$ and $c_0$ to have two distinct 2-paths connecting it to $x_1$. But this contradicts the planarity of 
$\overrightarrow{G}$ unless we have $n_y = 1$. Therefore, $n_y = 1$.

If $n_y = 1$, then $n_x \geq 4$ by equation~\ref{orientable order}. If we have the edge $xy$ (say, inside 
region $R_2$), then each vertex of $S_x$ must be adjacent to $c_0$ to have two distinct 2-paths connecting it 
to  $c_1$ creating the dominating set $\{c_0,x\}$ with at least four common neighbors (the vertices of $S_x$) contradicting the maximality of $D$. Hence we do not have the edge $xy$.

Therefore, by triangulation, we must have the edges $c_0c_1$ and $c_1c_2$. Now it is not possible to have more than two vertices of $S_x$ adjacent to $c_0$ as it will create the dominating set $\{c_0,x\}$ contradicting the maximality of $D$. Similarly, it is not possible to have more than two vertices of $S_x$ adjacent to $c_2$.
But as $|S_x| \geq 4$, by triangulation we must have (to avoid having more than two vertices from $S_x$ 
adjacent to $c_0$ or $c_2$) at least two vertices of $S_x$ adjacent to $y_1$. This will create 
the dominating set $\{y_1,x\}$ with at least four common neighbors ($c_0,c_2$ and two vertices of $S_x$) contradicting the maximality of $D$.

Hence it is not possible to have $|C| = 3$. 
\end{proof}

The above three lemmas prove that for no value of $|C|$ it is possible to have a planar push clique 
with domination number 2 and order at least 9.

\section{Proof of Theorem~\ref{th planar push clique list}}\label{sec list}
We want to show that any planar underlying push clique must contain one of the graphs listed in Fig.~\ref{fig planar push cliques} as a spanning subgraph. 
It is easy to observe that each graph  listed in Fig.~\ref{fig planar push cliques} is planar (we provided a planar drawing) and a push clique 
(it is easy to check). 
Note that by adding edges to a push clique one obtains another push clique. 
So to prove our result it is enough to show that if $\overrightarrow{G}$ is a planar push clique with $k$ vertices then 
its underlying graph $G$ must contain one of the graphs with $k$ vertices listed in Fig.~\ref{fig planar push cliques} as its subgraph. 
From now on in this section whenever we mention a graph $H_i$ for any $i \in \{1, 2,\cdots, 18 \}$, we mean the $i$th graph depicted  in Fig.~\ref{fig planar push cliques}. 

Let $\overrightarrow{G}$ be a planar push clique with underlying graph $G$. Note that, $G$ must be connected and any two non-adjacent vertices of $G$ must have at least two common neighbors. Thus, 

\begin{observation}\label{ob key}
Any underlying push clique that is not a complete graph must contain a 4-cycle as subgraph. 
\end{observation}

\medskip

If $|G| \leq 4$,  then by Observation~\ref{ob key} we can say that $G$ contains one of the four graphs $H_1, H_2, H_3, H_4$ as its spanning subgraph.   

\medskip

If  $|G| = 5$, then $G$ must contain a 4-cycle. The fifth vertex of $G$ must be adjacent  to at least two vertices of 
the 4-cycle. Thus,  $G$  either contains a $K_{2,3}$ or must contain a 5-cycle.   

It is easy to observe that $K_{2,3}$ is not an underlying push clique. If we replace the partite set containing
two vertices with an edge, then also the graph obtained is not an underlying push clique. But if we add an edge in the other partite set, a 5-cycle is created. Thus, $G$ must contatin a 5-cycle $abcdea$ (say).   

Note that the 5-cycle $abcdea$ with two incident chords $ac$ and $ce$ is not an underlying push clique
 as $b$ and $d$ are neither adjacent nor they have at least two common neighbors. 
A 5-cycle and a 5-cycle with a single chord is a subgraph of the above mentioned graph, thus not underlying push cliques. Therefore,
to obtain an underlying push clique from a 5-cycle we need  at least two non-incident edges.  
  The graph we get is $H_5$.
  
  \medskip

\begin{lemma}\label{ob key2}
If $H$ is an underlying planar push clique of order $n \geq 6 $ and has minimum degree 2, then $H$ is isomorphic to $H_6$.
\end{lemma}

  \begin{proof}
Let $H$ be an underlying planar push clique of order $n \geq 6 $. It is easy  to note that any vertex of $H$ must have degree at least two. 
 If some vertex $v$ of $H$ has degree equal to two then 
 each of its non-neighbor must be adjacent to both the neighbors of $v$, resulting in a $K_{2,n-2}$.  
 Given any orientation of $H$ we can push the neighbors of $v$, say $x$ and $y$, 
 in such a way that we have the arcs $\overrightarrow{xv}$ and $\overrightarrow{vy}$. 
 Thus, 
for being a push clique, each non-neighbor of $v$ must be in a special $4$-cycle with $v$ while the other vertices of the cycle are $x$ and $y$.  
Therefore, it is possible to push the non-neighbors of $v$ to obtain an  orientation of $H$ such that $\overrightarrow{xw}$ and $\overrightarrow{yw}$ are arcs for each non-neighbor $w$ of $v$. 
Let this so obtained  orientation of $H$ be $\overrightarrow{H}$. 
Now note that the only way for $\overrightarrow{H}$ to be a  planar push clique is to have the non-neighbors of $v$ induce a 2-dipath. 
In that case, $n-2 = 4$ and the underlying graph of $\overrightarrow{H}$ is isomorphic to the graph $H_6$. 
  \end{proof}

  \medskip

 If  $|G| = 6$, then either $G$ has minimum degree at least 3 or $G$ is isomorphic to  $H_6$.
In any case, $G$ has a Hamiltonian cycle (Dirac's Theorem (1952)~\cite{D.B.West}), say, $abcdefa$.

 Note that a six cycle can have two types of chords, namely, a \textit{long chord} connecting vertices at distance 3 and a \textit{short chord}
 connecting vertices at distance 2. We will  obtain a minimal planar underlying push clique on 6 vertices by adding chords to the 
 six cycle $abcdefa$ by case analysis. 
 
 If we do not add any long chord then we need to at least four short chords. The graph we obtain without adding any long chord is isomorphic to 
  $H_9$.  
  If we add exactly one long chord then we obtain a graph  isomorphic to  $H_7$ or $H_8$.  
  If we add exactly two long chords then we obtain a graph  isomorphic to 
  $H_6$. In fact, we do obtain some other graphs which contains one of the graphs $H_7, H_8, H_9$ as proper subgraphs and hence does not make it to our list.
  It is not possible to add three long chords keeping the graph planar.

\medskip

\begin{observation}\label{ob ham6}
Each edge of the graphs $H_{6}, H_{7}, \ldots, H_{14}$ is part of a Hamiltonian cycle.
\end{observation}

\medskip

If  $|G| = 7$, then each vertex of $G$ must have degree at least three by 
Lemma~\ref{ob key2}. If  $G$ has minimum degree 4 then $G$ is
 Hamiltonian by Dirac's Theorem (1952)~\cite{D.B.West}. Otherwise, let $v$ be a degree three vertex of $G$. Delete the vertex $v$ from $G$ and add the edges among its neighbors to obtain the graph $G'$. Note that as $G$ was an underlying planar push clique on 7 vertices, $G'$ must be an 
underlying planar push clique on 6 vertices. Thus, by what we have proved by now, $G'$ must contain one of  $H_{6}, H_{7}, H_{8}, H_{9}$
as its spanning subgraph. 
If that spanning underlying push clique of $G'$ contains one of the edges among the neighbors of $v$ then $G$ is Hamiltonian by 
Observation~\ref{ob ham6}. As the graphs $H_{7}, H_{8}, H_{9}$ have independence number 2, we will be done if $G'$ has one of these graphs as its spanning subgraph. The graph $H_{6}$
has independence number 3 and has exactly one independent set  of cardinality 3. 
If we add a vertex to the graph and make it adjacent to those three vertices, then a $K_{3,3}$ is created and thus the so obtained graph is not planar. Therefore, we can conclude that 
$G$ is Hamiltonian. After that a routine case analysis assuming number of long chords will settle this case. 

\medskip

If  $|G| = 8$, then each vertex of $G$ must have degree at least three by Lemma~\ref{ob key2}. If  $G$ has minimum degree 4 then $G$ is
 Hamiltonian by Dirac's Theorem (1952)~\cite{D.B.West}. 
The graphs $H_{11},H_{12},H_{13}$ have independence number 2. 
The graphs $H_{10}, H_{14}$ have independence number 3. Each of them has a unique independent set of  cardinality 3.   
If we add a vertex to $H_{10}$ or $H_{14}$  and make it adjacent to the vertices of its unique independent set of cardinality 3, then a $K_{3,3}$ minor is created and thus the so obtained graph is not planar. 
 Therefore, if $G$ has minimum degree 3, then using Observation~\ref{ob ham6} and arguing exactly like the case above we can conclude that  
$G$ is Hamiltonian.

Note that an eight cycle can have three types of chords, namely, a \textit{very long chord} connecting vertices at distance 4, a long chord and 
a short chord. 
A routine case analysis assuming number of very long chords with subcases assuming the number of long chords will settle this case.

\section{Conclusions}\label{sec conclusion}
We listed all  minimal planar underlying push cliques upto spanning subgraph inclusion. One can notice that there are four distinct minimal 
 planar underlying push cliques of maximum order (eight vertices). 
 This result is unlike the case of  planar underlying oriented clique where there is a unique planar underlying oriented clique of maximum order~\cite{oclique_nandy}. 
 
 The planar oriented cliques were instrumental in improving the bound of oriented chromatic number of planar graphs. Thus we hope that our list will help studies related to pushable chromatic number of oriented planar graphs.




\bibliographystyle{abbrv}
\bibliography{POreferences}

\begin{thebibliography}{10}

\bibitem{bensmail2013complexity}
J.~Bensmail, R.~Duvignau, and S.~Kirgizov.
\newblock The complexity of deciding whether a graph admits an orientation with
  fixed weak diameter (submitted).
\newblock {\em https://hal.archives-ouvertes.fr/hal-00824250/}, 2013.

\bibitem{fisher-push}
D.~C. Fisher and J.~Ryan.
\newblock Tournament games and positive tournaments.
\newblock {\em Journal of Graph Theory}, 19(2):217--236, 1995.

\bibitem{dom}
W.~Goddard and M.~A. Henning.
\newblock Domination in planar graphs with small diameter.
\newblock {\em Journal of Graph Theory}, 40:1--25, May 2002.

\bibitem{klostermeyer1}
W.~F. Klostermeyer.
\newblock Pushing vertices and orienting edges.
\newblock {\em Ars Combinatoria}, 51:65--76, 1999.

\bibitem{klostermeyer2}
W.~F. Klostermeyer et~al.
\newblock Hamiltonicity and reversing arcs in digraphs.
\newblock {\em Journal of Graph Theory}, 28(1):13--30, 1998.

\bibitem{36}
W.~F. Klostermeyer and G.~MacGillivray.
\newblock Analogs of cliques for oriented coloring.
\newblock {\em Discussiones Mathematicae Graph Theory}, 24(3):373--388, 2004.

\bibitem{push}
W.~F. Klostermeyer and G.~MacGillivray.
\newblock Homomorphisms and oriented colorings of equivalence classes of
  oriented graphs.
\newblock {\em Discrete Mathematics}, 274(1--3):161--172, 2004.

\bibitem{oclique_nandy}
A.~Nandy, S.~Sen, and {\'{E}}.~Sopena.
\newblock Outerplanar and planar oriented cliques.
\newblock {\em Journal of Graph Theory, doi: 10.1002/jgt.21893}, 2015.

\bibitem{EC_sen}
S.~Sen.
\newblock On push chromatic number of planar graphs and planar p-cliques.
\newblock In J.~Ne{\v{s}}et{\v{r}}il and M.~Pellegrini, editors, {\em The
  Seventh European Conference on Combinatorics, Graph Theory and Applications},
  volume~16 of {\em CRM Series}, pages 617--617. Scuola Normale Superiore,
  2013.

\bibitem{sopena_updated_survey}
{\'{E}}.~Sopena.
\newblock Homomorphisms and colourings of oriented graphs: An updated survey,
  doi: 10.1016/j.disc.2015.03.018.
\newblock {\em Discrete Mathematics}, 2015.

\bibitem{D.B.West}
D.~B. West.
\newblock {\em Introduction to Graph Theory ($2^{nd}$ Edition)}.
\newblock Prentice Hall, 2001.

\end{thebibliography}

%
%
%
 
\end{document}